\documentclass{amsart}
\usepackage{braket}
\usepackage{amssymb}
\usepackage[english]{babel}
\usepackage{enumerate}
\usepackage{bussproofs,mathdots}

\usepackage{proof}
\newcommand{\len}{{\rm len}}
\newtheorem{theorem}{Theorem}[section]
\newtheorem{proposition}[theorem]{Proposition}
\newtheorem{lemma}[theorem]{Lemma}

\newcommand\dual[1]{\check{#1}}
\newcommand\dom{\mathop\mathrm{dom}}
\newcommand\FV{\mathop\text{FV}}

\newcommand\nmodels\nvDash
\renewcommand\models\vDash
\newcommand\FO{\mathrm{FO}}
\newcommand\df{\mathrm{D}}
\newcommand\M{\mathbb{M}}
\def\dep{\mathord=}
\renewcommand\varphi\phi
\renewcommand\overline\bar
\newcommand\imp\rightarrow
\renewcommand\vec\bar
\renewcommand\H{\mathcal{H}}

\begin{document}

\title{Dependence Logic with Generalized Quantifiers: Axiomatizations} 
\author{Fredrik Engstr\"om \and Juha Kontinen \and Jouko V\"a\"an\"anen}
\date{\today}

\begin{abstract}
	We prove two completeness results, one for the extension of dependence logic by a monotone
	generalized quantifier $Q$ with weak interpretation, weak in the meaning that the interpretation of 
	$Q$ varies with 
	the structures. The second result considers the extension of dependence logic where $Q$ is
	interpreted as 
	``there exists uncountable many.'' Both of the axiomatizations are shown to be sound and complete for
	$\FO(Q)$ consequences.
\end{abstract}

\maketitle

\section{Introduction}

Generalized quantifiers constitute a well-studied method of extending the expressive power of first order logic. A more recent extension of first order logic is obtained by adding dependence atoms, permitting the expression of partially ordered quantification. In this paper we study the combination of the two methods, adding to first order logic both generalized quantifiers and dependence atoms as defined  in  \cite{Engstrom:2011}. It was shown in  \cite{Engstrom.Kontinen:2013} that the resulting extension properly extends both the respective extension  by generalized quantifiers and the extension by dependence atoms. We analyse further the expressive power and give natural axioms for the new logic. There are theoretical limits to the extent that the axioms can be complete but we give partial completeness results in the sense that completeness is shown with respect to $\FO(Q)$ consequences. 

Generalized quantifiers were introduced by Mostowski \cite{Mostowski:1957}. The most important of them  is perhaps the quantifier
\begin{equation*}
\M\models Q_1x\phi(x,\vec{b})\iff \M\models \phi(a,\vec{b})\mbox{ for uncountably many $a\in M$}
\end{equation*}
owing to the beautiful axiomatization of it by Keisler \cite{keisler1970logic}. On the other hand, generalized quantifiers have made an entrance to both linguistics \cite{Peters2006-PETQIL} and computer science \cite{MR1336413}. In natural language  we can use generalized quantifiers to analyse constructs such as
\begin{equation*}
\label{nl}\mbox{Most boys run,}
\end{equation*}
where we think of ``most" as a generalized quantifier. Other natural language quantifiers are ``two thirds", ``quite a few", ``many", etc. There are various so-called polyadic lifts of such quantifiers, for example,
\medskip

\noindent Ramsey lift:
$$\mbox{ At least two thirds of the boys in your class like each other. 
\cite{MR1457200}}$$

\noindent Branching lift:
$$\mbox{ A few boys in my class and a few girls in your class}\mbox{ have dated each other. \cite{MR527403}} $$

\noindent Resumption lift:
$$\mbox{ Most neighbours like each other. \cite{MR1457200}}$$

In computer science, or more exactly finite model theory, we have the {\em counting quantifiers}
\begin{equation*}
\label{cnt}\M\models \exists_{\ge k}x\phi(x,\vec{b})\iff \M\models  \phi(a,\vec{b})\mbox{ for at least $k$ elements $a\in M$}
\end{equation*}
which, although first order definable, have turn out relevant, but also the non-first order
\begin{equation*}
\label{evn}\M\models \exists_{\mbox{\tiny even}}x\phi(x,\vec{b})\iff \M\models\phi(a,\vec{b})\mbox{ for an even number of $a\in M$}
\end{equation*}
and \begin{equation*}
\label{half}\M\models \exists_{n/2}x\phi(x,\vec{b})\iff \M\models  \phi(a,\vec{b})\mbox{ for at least 50\%  of the elements  $a\in M$}
\end{equation*}
and other similar ones. The lifts, also called vectorizations, are important in finite model theory, too. For example, the resumption lift sequence of the transitive closure quantifiers characterises in ordered models NLOGSPACE \cite{DBLP:journals/siamcomp/Immerman87}, and the resumption lift sequence of the so-called alternating transitive closure quantifier characterises, even in unordered models, least fixpoint logic \cite{DBLP:conf/birthday/Dahlhaus87}.  

Dependence logic was introduced in \cite{Vaananen:2007}. It gives compositional semantics to the partially ordered quantifiers of \cite{Henkin:1961}:
$$\left(\begin{array}{ll}
\forall x&\exists y\\
\forall u&\exists v
\end{array}\right)\phi(x,y,u,v,\vec{z})\iff\exists f\exists g\forall x\forall u\phi(x,f(x),u,g(u),\vec{z}).$$
The compositional analysis is 
$$\left(\begin{array}{ll}
\forall x&\exists y\\
\forall u&\exists v
\end{array}\right)\phi(x,y,u,v,\vec{z})\iff
\forall x\exists y\forall u\exists v(\dep(\vec{z},u,v)\wedge\phi(x,y,u,v,\vec{z})),$$ where $\dep(\vec{z},u,v)$ is a so-called {\em dependence atom}.
Dependence logic has the same expressive power as existential second order logic \cite{Kontinen:2009}. Thus dependence logic alone cannot express, for example, uncountability, in fact not even finiteness.

The idea of combining partially ordered quantifiers and generalized quantifiers was first suggested by 
Barwise \cite{MR527403}. He used this combination to analyse  lifts such as the Ramsey lift and the branching lift. It was proved in  \cite{MR1457200} that the polyadic lifts of monotone (unbounded) generalized quantifiers lead to a strong hierarchy, giving immediately the result that there is no finite number of generalized quantifiers, including partially ordered quantifiers, which would be able to express all Ramsey lifts of a given monotone (unbounded) quantifier. The same is true of branching lifts, and to a lesser extent of resumption lifts \cite{MR1457200}. 

The situation is quite different with the extension of {\em dependence} logic (rather than {\em first order} logic) by a monotone generalized quantifier. All the mentioned polyadic lifts (and vectorizations) can be readily defined (in all arities). Let us see how this is done for the Ramsey lift of a monotone quantifier $Q$. 
 $$\exists A\in Q\forall x\in A\forall y\in A\phi(x,y,\vec{z})$$
 $$\iff$$
$$\begin{array}{ll}
\exists w(\dep(\vec{z},w)\wedge Q x\exists y(&y=w\ \wedge\\
&\forall u\exists v(\dep(\vec{z},u,v)\wedge(x=u\to v=w)\wedge\\
&\forall u'\exists v'(\dep(\vec{z},u',v')\wedge(u=u'\to v'=w)\wedge\\
&((v=w\wedge v'=w)\to\phi(u,u',\vec{z}))))))\end{array}$$ 
Respectively, the branching lift can be expressed as follows:

 $$\exists A\in Q\exists B\in Q\forall x\in A\forall y\in B\phi(x,y,\vec{z})$$
 $$\iff$$
$$\begin{array}{ll}
\exists w,w'(&\dep(\vec{z},w)\wedge\dep(\vec{z},w')\wedge\\
&Q x\exists y(y=w\wedge\dep(\vec{z},x,y)\ \wedge\\
&Q x'\exists y'(y'=w'\wedge\dep(\vec{z},x',y')\ \wedge\\
&\forall u\exists v(\dep(\vec{z},u,v)\wedge(x=u\to v=w)\wedge\\
&\forall u'\exists v'(\dep(\vec{z},u',v')\wedge(x'=u'\to v'=w')\wedge\\
&((v=w\wedge v'=w')\to\phi(u,u',\vec{z}))))))\end{array}$$ 
Resumption can be handled similarly.

Thus putting generalized quantifiers and dependence atoms together results in a powerful combination extending far beyond either generalized quantifiers alone or dependence atoms alone.   

This paper is organised as follows. In Section \ref{Prel} we review the basics on dependence logic
and  generalized quantifiers in the dependence logic context. In Section \ref{Ded} we present a
system of  natural deduction  for the extension $\df(Q,\dual Q)$  of dependence logic by a
monotone generalized quantifier $Q$ and its dual $\dual Q$, and show that these rules are sound. Finally in Section \ref{Comp} 
two completeness results for $\FO(Q)$ consequences  are shown for $\df(Q,\dual Q)$.  In  the first  result $Q$ has the so-called weak interpretation, and in the second $Q$ is  interpreted as $Q_1$, that is, ``there exists uncountable many.'' 
	
%Another attempt of introducing generalized quantifiers
%in a framework of team semantics can be found in Kuusisto's paper \cite{kuusisto2012defining}.
%Some part on $Q_1$. And weak models. \cite{keisler1970logic, kaufmann1985quantifier} 
%Explaining our results. (and how to get them with the "soft" argument)

\section{Preliminaries}\label{Prel}

\subsection{Dependence Logic}

In this section we give a brief introduction to dependence logic. For a 
detailed account see \cite{Vaananen:2007}.

The syntax of dependence logic extends the syntax of first order logic with 
new atomic formulas, the dependence atoms. There is one
dependence atom for each arity. We write the atom expressing that the term 
$t_n$ is uniquely determined by the values of the terms $t_1,\ldots,t_{n-1}$ 
as
$\dep(t_1,\ldots,t_n)$. We consider formulas where negation can only appear in front of formulas
without dependence atoms. For a vocabulary $\tau$, $\df[\tau]$ denotes the set of 
$\tau$-formulas of dependence logic. 
The set $\FV(\varphi)$ of free variables of a formula $\varphi$ is defined as in first order logic except that
$$\FV(\dep(t_1,\ldots,t_n)) = \FV(t_1) \cup \ldots \cup \FV(t_n).$$

To define a compositional semantics for dependence logic we use \emph{sets of 
	assignments}, called \emph{teams}, instead of single assignments as in 
first order logic. An assignment
is a function $s: V \to M$ where $V$ is a finite set of variables and
$M$ is the universe under consideration. A team on $M$ is 
a set of assignments for some fixed finite set of variables $V$. If 
$V=\emptyset$ there is only one assignment, the empty function 
$\emptyset$. Observe that the team of the empty assignment $\set{\emptyset}$ is 
different from the empty team $\emptyset$.

\begin{itemize}
	\item Given an assignment $s: V \to M$ and $a \in M$ let $s[a/x]: V \cup \set{x} \to 
M$ be the assignment:
$$
s[a/x]: y \mapsto  \begin{cases}
s(y)  &\text{ if $y \in V \setminus \set{x}$, and}\\
a  &\text{ if $x=y$.}
\end{cases}
$$
\item Let $X[M/y]$ be the team $$\set{s[a/y] | s \in X, a \in M},$$
\item and whenever $f: X \to M$, let $X[f/y]$ denote the team $$\set{s[f(s)/y] | s \in X}.$$
	\end{itemize}
	The domain of a non-empty team $X$, denoted $\dom(X)$, is the set of variables 
$V$.
 The interpretation of the term $t$ in the model $\M$ under the assignment $s$ 
 is denoted by $t^{\M,s}$. We write $s(\overline{x})$ for the tuple obtained by pointwise
 application  of $s$ to the finite sequence $\bar x$ of variables.

The satisfaction relation for dependence logic $\M ,X \models  \varphi$ is
defined as follows.  Below, the notation $\M,s \models \varphi$ refers to the 
ordinary satisfaction relation of first order logic. We also assume that $\FV(\varphi) \subseteq
\dom(X)$.

\begin{enumerate}
	\item For formulas $\psi$ without dependence atoms: $ \M ,X 
	\models  \psi \text{ iff } \forall s \in X: \M,s \models \psi$.
\item $\M ,X \models  \dep(t_1,\ldots,t_{n+1}) \text{ iff }
\forall s,s' \in X :\bigwedge_{1\leq i \leq n}t_i^{\M,s}
=t_i^{\M,s'}  \rightarrow  t_{n+1}^{\M,s} =t_{n+1}^{\M,s'}$
%\item $\M ,X\models \lnot \dep(t_1,\ldots,t_{n+1}) \text{ iff } X = \emptyset $
\item $\M ,X \models  \varphi \land \psi \text{ iff }\M ,X \models \varphi 
	\text{ and } \M ,X \models  \psi$
\item $\M ,X \models  \varphi \lor \psi \text{ iff there are $Y$ and $Z$ such that  
	} X= Y \cup Z,\text{ and } \M,Y \models  \varphi \text{ and } \M,Z 
	\models \psi$
\item $\M ,X \models  \exists y \varphi \text{ iff there is } f: X \to M, 
	\text{ such that } \M,{X[f/y]}
\models \varphi $
\item $\M ,X \models  \forall y \varphi \text{ iff }  \M,X[M/y]
\models \varphi.$
\end{enumerate}

We define $\M \models \sigma$ for a sentence $\sigma$ to hold if $\M , 
{\set{\emptyset}} \models \sigma$.
Let us make some easy remarks.
\begin{itemize}
	\item Every formula is satisfied by the empty team. 
	\item The satisfaction relation is downwards closed:
If $\M ,X \models  \varphi$ and $Y \subseteq X$ then $\M ,Y \models  \varphi$.  
\item The satisfaction relation is local:
$\M,X \models \varphi$ iff $\M,Y\models \varphi$ where $$Y=\set{s 
	\upharpoonright \FV(\varphi) |  s \in X}.$$
\end{itemize}

The expressive power for sentences 
of dependence logic is the same as that of existential second order logic.

\subsection{D(Q)}

The notion of a generalized quantifier goes back to
Mostowski \cite{Mostowski:1957} and Lindstr\"om \cite{Lindstrom:1966}. In 
\cite{Engstrom:2011} semantics for generalized
quantifiers in the framework of dependence logic was introduced. We will review the
definitions below.

Let $Q$ be a quantifier of type $\langle k \rangle$, meaning
that $Q$ is a class of $\tau$-structures, where the signature $\tau$ has a
single $k$-ary relation symbol.  Also, assume that $Q$ is monotone 
increasing, i.e.,  for every $M$ and every $A \subseteq B \subseteq M^k$, if 
$A \in Q_M$ then also $B \in Q_M$, where $Q_M= \set{R \subseteq M^k | (M,R) \in 
Q}$.

The formulas of dependence logic extended with a quantifier $Q$, $\df(Q)$, is built up from 
$\FO(Q)$-formulas and dependence atoms using the connectives $\land$ and $\lor$, and the
quantifier expressions $\exists x$, $\forall x$ and $Qx$ in the usual way. We write $\phi\rightarrow \psi$
 as a shorthand for $\neg \phi\vee \psi$, where $\phi$ is a formula without dependence atoms.

An assignment $s$ satisfies a formula
$Q\bar x\, \varphi$ in a structure $\mathbb M$,
$$\mathbb M,s \models 
Q\bar x\,\varphi,\text{ if the set } \set{\bar a \in M^k | \mathbb M,s[\bar a /\bar x] 
\models \varphi} \text{ is in }Q_M.$$ 

In the context of teams we say that a team $X$ satisfies a formula $Q\bar x\, 
\varphi$,
\begin{equation}\label{def_q}
\mathbb M ,X \models  Q\bar x\, \varphi \text{, if there exists } F: X \to Q_M
\text{ such that } \mathbb M,{X[F/\bar x]} \models \varphi,
\end{equation}
where $X[F/\bar x] = \set{s[\bar a /\bar x] | \bar a \in F(s)}$.  
This definition works well only with monotone (increasing) quantifiers, see 
\cite{Engstrom:2011} for details.

The following easy proposition suggests that we indeed have the right truth 
condition for monotone quantifiers:

\begin{proposition}[\cite{Engstrom:2011}]\label{prop:4} 
	\begin{enumerate}[(i)]
\item\label{en1} $\df(Q)$ is downwards closed.
\item\label{en2} $\df(Q)$ is local, in the sense that $\M,X \models \varphi$ iff 
	$\M,(X \mathbin\upharpoonright \FV(\varphi)) \models \varphi$.
\item\label{en3} Viewing $\exists$ and $\forall$ as generalized quantifiers of 
	type $\langle 1 \rangle$, the truth conditions in \eqref{def_q} are 
	equivalent to the truth conditions of dependence logic.
%\item\label{en4} For $\df(Q)$-formulas without dependence atoms $\varphi$ and teams $X$, $\mathbb M ,X 
%	\models  \varphi$ iff
%for all $s \in X: \mathbb M ,s \models  \varphi$.
\item\label{en5} For every $\df(Q)$ formula $\varphi$ we have $\M,\emptyset 
	\models \varphi$.
\end{enumerate}
\end{proposition}

As proved in \cite{Engstrom.Kontinen:2013}, the expressive power of $\df(Q)$ sentences corresponds to that of 
a  certain natural extension of existential second order logic by $Q$.

In order to get a prenex normal form for all formulas we will focus on the logics $\df(Q,\dual
Q)$,  where $\dual Q$ is the dual of
$Q$, i.e, $$\dual Q = \set{(M,M^k \setminus R) | R \subseteq M^k ,(M,R) \notin Q},$$ 
instead of $\df(Q)$. Note that, according to our definition of $\df(Q)$, a formula $Qx \phi$ may be negated only if
$\phi$ is a $\FO(Q)$ formula.

We will  consider monotone increasing quantifiers $Q$ satisfying two 
non-triviality assumptions:  $(M,\emptyset)\notin Q$ and  $(M,M^k)\in Q$ for 
all $M$. In  \cite{Engstrom.Kontinen:2013} the following normal form for sentences of 
 $\df(Q)$ was shown for such non-trivial quantifiers.

\begin{theorem}
	Every $\df(Q)$ sentence in negation normal form, where $Q$ is non-trivial, can be written as 
$$
\H^1 \bar{x}_1\ldots \H^m \bar{x}_m \exists y_1\ldots \exists y_n \bigl(\bigwedge _{1\le 
	j\le
n}\dep(\overline{z}^i,y_i)\wedge\theta\bigr),
$$
where $\H^i$ is either $Q$ or $\forall$ and $\theta$ is a quantifer-free $\FO$-formula.
\end{theorem}

In the present paper a similar normal form for all $\df(Q,\dual Q)$ formulas is obtained in
Proposition \ref{prop:nf}. 

A \emph{weak semantics} can be given for $\df(Q)$ (and $\FO(Q)$, etc) by regarding $Q$ as an
interpreted symbol rather than a logical constant in the following way (see
\cite{keisler1970logic} and \cite{kaufmann1985quantifier} for more on this). A weak model is a
structure together with an interpretation of $Q$, often denoted by $q$. We define $T \models_w
\sigma$ to hold if every weak model $(M,q)$ of $T$ satisfies $\sigma$. In this paper we require the
interpretation $q$ of $Q$ to be monotone increasing and non-trivial (In essence this is the
\emph{monotone logic} of \cite{makowsky1977some}). In the weak semantics for $\df(Q,\dual Q)$ we
require that the interpretation $\dual q$ of $\dual Q$ is the dual of the
interpretation $q$ of $Q$. Thus, if $T \cup \set{\sigma}$ consists of  $\df(Q,\dual Q)$ sentences, then
$T \models_w \sigma$ if every model $(M,q,\dual q)$ of $T$ satisfies $\sigma$.

\section{Natural deduction for $\df(Q,\dual Q)$}\label{Ded}

In this section we present a set of natural deduction rules for the logic $\df(Q,\dual Q)$, where $Q$
is monotone and satisfies the non-triviality conditions: $(M,\emptyset)\notin Q$ and  $(M,M^k)\in
Q$ for all $M$. 
Observe that then also the dual quantifier $\dual Q$ satisfies these conditions. 
To simplify notation, we will restrict attention to type $\langle 1\rangle$ quantifiers. 

We use an abbreviation $\vec{x}=\vec{y}$ for the formula  $\bigwedge_{1\le i \le \len(\vec{x})} x_i=y_i$,
assuming of course  that  $\vec{x}$ and $\vec{y}$ are tuples  of the same length $\len(\vec{x})$.
The  substitution  of a term $t$ to the free occurrences of $x$ in $\psi$ is denoted by
$\psi[t/x]$. Analogously to first order logic,  no variable of $t$ may become bound in such a substitution.  For tuples  $\bar{t}=(t_1,\ldots,t_n)$ and $\bar{x}=(x_1,\ldots,x_n)$ we write   
$\psi[\bar{t}/\bar{x}]$ to denote the simultaneous substitution of $x_i$ by $t_i$ for $1\le i \le n$. 

Here is a list of all the rules:

\begin{enumerate}
	\item Conjunction:$$
		\infer[{\mbox{$\wedge$I}}]{\phi \wedge \psi}{\phi & \psi}
\qquad \infer[{\mbox{$\wedge$E}}]\phi{\phi\wedge \psi}
\qquad
\infer[{\mbox{$\wedge$E}}]\psi{\phi\wedge \psi}$$

\item Disjunction:
$$
\infer[{\mbox{$\vee$I}}]{\phi \vee \psi}{\phi}
\qquad
\infer[{\mbox{$\vee$I}}]{\phi \vee \psi}{\psi}
\qquad
\infer[{\mbox{$\vee$E}}]{\gamma}{
       \phi\vee \psi
       &
                                   \infer*{\gamma}{
               [\phi]
       }
       &
      \infer*{\gamma}{
               [\psi]
       }
       }
$$
where $\gamma$ is a $\FO(Q,\dual Q)$ formula.

\item\label{rule:neg} Negation and duality:
$$
\infer[{\mbox{ $\lnot$I}}]{\lnot\phi}{
     \infer*{\bot}{
               [\phi]
       }
       }
 \qquad
\infer[{\mbox{RAA}}]{\phi}{
      \infer*{\bot}{
               [\lnot \phi]
       }
       }
			 \qquad \infer[{\mbox{ $\bot$I}}]{\bot}{\phi & \lnot \phi}
\qquad
\infer{\lnot Qx \lnot \phi}{\dual Qx\phi}
 $$
where $\phi$ is a $\FO(Q,\dual Q)$ formula.

\item Universal quantifier:
$$
\infer[{\mbox{$\forall$I}}]{\forall x_i\phi}{ \phi }
\qquad
\infer[{\mbox{$\forall$E}}]{\phi[t/x_i]}{\forall x_i \phi}
$$
In $\forall$I the variable \(x_i\) cannot appear free  in any non-discharged assumption
used in the derivation of $\phi$.

\item Existential quantifier:
$$
\infer[{\mbox{$\exists$I}}]{\exists x_i\phi}{
	\phi(t/x_i)
       }
\qquad
\infer[{\mbox{$\exists$E}}]{\psi}{
       \exists x_i\phi
       &
                                   \infer*{\psi}{
               [\phi]
       }
       }$$
In $\exists$E the variable \(x_i\) cannot appear free  in \(\psi\) and in any non-discharged assumption
used in the derivation of  \(\psi\), except in \(\phi\).

\item\label{rule1} Disjunction substitution: 
\[
\infer[]{\phi\vee \gamma}{
        \phi\vee \psi
       &
                                   \infer*{\gamma}{
               [\psi]
       }
       }\]

\item\label{rule2}  Commutation and associativity of disjunction: 
\[\infer{\phi\vee \psi}{\psi\vee \phi}\hspace{19mm}\infer{\phi\vee(\psi\vee \gamma)}{(\phi\vee
\psi)\vee \gamma}\]

\item\label{scopeQ}  Extending scope:
%\[\infer{ Q x (\phi \vee \psi)}{Q x \phi \vee \psi}\hspace{14mm} \infer{ Q x (\phi \wedge \psi)}{Q
%x \phi \wedge \psi}  \]
% where the  prerequisite for applying these rules is that $x$ does not appear free in $\psi$.
 %Similar rule for $\dual Q$.
\[\infer{ \mathcal{H} x (\phi \vee \psi)}{\mathcal{H} x \phi \vee \psi}\hspace{14mm} \infer{ Q x (\phi \wedge \psi)}{Q
x \phi \wedge \psi}  \]
 where $\mathcal{H}\in \{Q,\dual Q, \exists, \forall\}$, and  the  prerequisite for applying these rules is that $x$ does not appear free in $\psi$.
The rule on the right is also assumed for $\dual Q$.

\item\label{rule5} Unnesting:
\[\infer{\exists z( \dep(t_1,...,z,...,t_n)\wedge z=t_i)}{\dep(t_1,...,t_n)}   \]
 where $z$ is a new variable.

\item\label{rule6}  Dependence distribution: let  
\begin{eqnarray*}
\phi &=& \exists y_1\ldots \exists y_n(\bigwedge_{1\le j\le n}\dep(\vec{z}^j,y_j)\wedge \phi_0),\\
\psi &=& \exists y_{n+1}\ldots \exists y_{m}(\bigwedge_{n+1\le j\le
m}\dep(\vec{z}^j,y_j)\wedge \psi_0).
\end{eqnarray*}
where $\phi_0$ and $\psi_0$ are quantifier-free formulas without dependence atoms, and $y_i$, for
$1\le i \le n$, does not appear in $\psi$ and $y_i$, for $n+1\le i \le m$, does not appear in
$\phi$.
Then,
\[\infer{\exists y_1 \ldots \exists y_{m}(\bigwedge_{1\le j\le m}\dep(\vec{z}^j,y_j)\wedge
(\phi_0\vee \psi_0))}{\phi\vee \psi}   \]

\item\label{rule:depi} 
Dependence introduction:
\[\infer{\forall y\exists x (\dep(\vec{z},x)\wedge \phi)}{\exists x \forall y \phi}\hspace{14mm}
\infer{Q y\exists x (\dep(\vec{z},x)\wedge \phi)}{\exists x Q y \phi}  \]
 where $\vec{z}$ lists the variables in $\FV(\phi)- \{x,y\}$. Similar for $\dual Q$.

\item \label{rule:mon} Monotonicity of $Q$ and $\dual Q$: 
\[
\infer[]{Qx\psi}{
        Qx\phi
       &
                                   \infer*{\psi}{
               [\phi]
       }
       }\]
where the  prerequisite for applying this rule is that the variable $x$ cannot appear free  in any non-discharged assumption
used in the derivation of $\psi$, except for $\phi$. 
Similar for $\dual Q$.

\item\label{rule:bound} Bound variables: 
	$$	\infer{Qy\phi[y/x]}{Qx\phi},$$
where $y$ does not appear in $\phi$.	 
Similar for $\dual Q$. 

\item Identity rules:
	$$ \infer{t=t}{} \qquad
	\infer{\phi[t/x]}{\phi[r/x] & t=r}$$
	where $\phi$ is an $\FO(Q,\dual Q)$ formula. 

%\item Duality:
%	$$\infer{\lnot Qx \lnot \phi}{\dual Qx\phi},$$
%	where $\phi$ is an $\FO(Q,\dual Q)$ formula. 
\end{enumerate}

Observe that $\FO(Q,\dual Q) \equiv \FO(Q)$, but syntactically $\FO(Q,\dual Q)$ includes more formulas. 

\subsection{Soundness of the rules}

In this section we show the soundness of the rules introduced in the previous section under any monotone and non-trivial
interpretation of $Q$.  
Clearly this is the same as
soundness in the weak
semantics for $Q$. 

The following lemmas will be needed in the proof.

\begin{lemma}\label{terms} Let $\phi(x)$ be a $\df(Q,\dual Q)$ formula, and $t$ a term such that
	in the substitution $\phi[t/x]$ no variable of $t$ becomes bound. Then for all $\M$ and teams
	$X$, where $(\FV(\phi)-\{x\})\cup \mathrm{Var}(t)\subseteq \dom(X)$\[  \M,X\models  \phi[t/x]
	\Leftrightarrow   \M,X[F/x]\models \phi(x),  \]
where $F\colon X\rightarrow A$ is defined by  $F(s)= t^{\M,s}$.
\end{lemma}
\begin{proof} Analogous to Lemma 8 in   \cite{Kontinen:2012}.
\end{proof}

It is easy to verify that Lemma \ref{terms} gives the following familiar property concerning changing free variables. 

\begin{lemma}[Change of free variables]\label{freevariables}
Let the free variables of $\phi\in  \df(Q,\dual Q)$  be $x_1,\ldots,x_n$ and let $y_1,\ldots,y_n$ be distinct variables. Then for all structures $\M$ and teams $X$ with domain $\{x_1,\ldots,x_n\}$ it holds that 
\[ \M,X\models \phi \Leftrightarrow \M,X'\models \phi[\bar{y}/\bar{x}],   \] 
where $X'$ is the team with domain  $\{y_1,\ldots,y_n\}$  containing the assignments $s'\colon y_i\mapsto s(x_i)$ for $s\in X$.
\end{lemma}

\begin{proposition}\label{soundness} Assume that $Q$ is monotone and non-trivial. 
	Let $T\cup\{\phi\}$ be a set of sentences of $\df(Q,\dual Q)$.  If $T\vdash \phi$, then $T
	\models \phi$.
\end{proposition}

\begin{proof}
	We prove the statement that if $T \vdash \phi$, where $T \cup \set{\phi}$ is a set
	of \emph{formulas}, then for any $\M$ and $X$ where $\dom(X) \supseteq \FV(T) \cup \FV(\phi)$, if $\M,X
	\models T$ then $\M,X \models \phi$. This is done by using induction on the length of derivation. 

It suffices to
consider the rules  \ref{rule:neg} (only duality), \ref{scopeQ}, \ref{rule:depi}, \ref{rule:mon}, and \ref{rule:bound} since the soundness of 
the other rules can be  proved analogously to  \cite{Kontinen:2012} using the fact that
$\df(Q,\dual Q)$ is local and has downwards closure (see (ii) and (i) of  Proposition \ref{prop:4}). In particular,  
 Lemma \ref{terms} is used in the soundness proofs of the rules $\exists$ I and $\forall$ E.

\begin{itemize}
\item[\eqref{rule:neg}] Assume $\M,X \models \dual Q x \phi$ then, since $\dual Qx\phi$ is a $\FO(Q,\dual Q)$ formula we
	have $\M,s \models \dual Qx \phi$ for all $s \in X$. This clearly implies that $\M,s \models \lnot
Qx \lnot \phi$ for all $s \in X$, which is equivalent to $\M,X \models \lnot Qx \lnot \phi$.

%\item[3]  Assume that we have a natural deduction proof of a formula $\lnot\phi$, where $\phi$ has no
%	dependence atoms, from the assumptions 
%\begin{equation*}
% \{\psi_1,\ldots,\psi_k\}
%\end{equation*}
%with the last rule 3.  Let $\M$ and $X$ be such that  $\M,X\models \psi_i$, for $1\le i\le k$. By
%the assumption, we have a shorter deduction of  $\bot$ from the assumptions $
%\{\phi,\psi_1,\ldots,\psi_k\}$. We claim that now $\M,X\models \lnot\phi$, i.e., $\M,s\models
%\lnot\phi$
%for all $s\in X$. For contradiction, assume that  $\M,s\not\models  \lnot\phi$ for some $s\in X$.
%Then $\M,s\models \phi$, and since $\phi$ is a formula without dependence atoms we get that
%$\M,\set{s} \models \phi$. By  downward closure, $\M,\set{s}\models \psi_i$, for $1\le
%i\le k$. Thus, by the induction assumption, we get $\M,\set{s}\models \bot$ which is a contradiction.

\item[\eqref{scopeQ}] These rules preserve logical equivalence analogously to  Lemma  3.2 in \cite{Engstrom.Kontinen:2013}.

\item[\eqref{rule:depi}] The soundness of this rule follows from the logical equivalence 
\[Q y\exists x (\dep(\vec{z},x)\wedge \phi)\equiv \exists x Q y \phi\]
the proof of which is analogous to the case where $Q$ is replaced by $\forall$  (see \cite{Kontinen:2012}).

\item[\eqref{rule:mon}]  Assume that we have a natural deduction proof of $Qx\psi$ from the assumptions 
\begin{equation*}
 \{\gamma_1,\ldots,\gamma_k\}
\end{equation*}
with the last rule 13.  Let $\M$ and $X$ be such that  $\M,X\models \phi_i$, for $1\le i\le k$. By
the assumption, we have a shorter deduction of  $Qx\phi$ from the assumptions $
\{\gamma_{n_1},\ldots,\gamma_{n_l}\}$ and a deduction of $\psi$ from 
the assumptions $ \{\phi,\gamma_{n_{l+1}},\ldots,\gamma_{n_m}\}$. Hence by the induction assumption it holds that  
$\M,X\models Qx\phi$. Therefore, there is $F\colon X\rightarrow Q_M$ such that $\M, X[F/x]\models \phi$. Since the variable
$x$ cannot appear free in the formulas $\gamma_{n_{l+1}}, \ldots, \gamma_{n_m}$ it follows that
$\M, X[F/x]\models \gamma_i$, for $i \in \set{n_{l+1},\ldots,n_m}$.    Now by the induction assumption we get that 
 $\M, X[F/x]\models \psi$ and   $\M, X\models Qx\psi$.

\item[\eqref{rule:bound}] This rule preserves logical equivalence by Lemma \ref{freevariables}. \qedhere
 \end{itemize}
\end{proof}

Note that since Proposition \ref{soundness} holds for every monotone non-trivial quantifier $Q$ we
get also soundness for weak semantics: If $T \vdash \phi$ then $T \models_w \phi$. 

\section{Completeness results for $\FO(Q,\dual Q)$ consequences}\label{Comp}

\subsection{Deriving a normal form for $\df(Q,\dual Q)$} 
In this section we show that from each formula $\phi \in\df(Q,\dual Q)$ we can derive a logically equivalent formula in the following normal form:
\begin{equation}\label{NF}
\mathcal{H}^1 x_1\ldots \mathcal{H}^m x_m \exists y_1\ldots \exists y_n \bigl(\bigwedge _{1\le 
	j\le
n}\dep(\overline{x}^i,y_i)\wedge\theta\bigr),
\end{equation}
where $\mathcal{H}^i$ is either $Q$, $\dual Q$ or $\forall$, and $\theta$ is a quantifier-free $\FO$-formula.

\begin{proposition}\label{prop:nf}Let $\phi$ be a formula of $\df(Q,\dual Q)$. Then $\phi\vdash \phi'$, where $\phi'$ is of the form \eqref{NF}, and $\phi'$ is logically equivalent to $\phi$. 
\end{proposition}
\begin{proof}
The proof of this Proposition is analogous to the proof of the corresponding result for dependence logic formulas in  \cite{Kontinen:2012}. We will indicate how the proof of \cite{Kontinen:2012} can be extended for the formulas of  $\df(Q,\dual Q)$.

We will establish the claim in several steps. Without loss of generality, we assume that in $\phi$ each variable is quantified only once and that, in  the dependence atoms of $\phi$, only variables (i.e. no complex terms) occur.  

\begin{itemize}
	\item[Step 1.] We derive from $\phi$ an equivalent sentence in prenex normal form:
\begin{equation}\label{2}
  \mathcal{H}^1 x_1\ldots \mathcal{H}^mx_m\chi,
\end{equation}
where $\mathcal{H}^i\in \{\exists,\forall, Q, \dual Q\}$ and $\chi$ is a quantifier-free formula. 

We will  prove the claim for every formula $\phi$  satisfying the assumptions made in the beginning of the proof and the assumption (if $\phi $ has free variables) that no variable appears both free and bound in $\phi$.  It suffices to consider the case $\phi:= \psi\vee \theta$, since the  case of conjunction is analogous and the other cases are trivial.

 By the induction assumption, we have derivations  $ \psi\vdash\psi^*$ and $\theta\vdash \theta^* $, where
\begin{eqnarray*}
\psi^* &=& \mathcal{H}^1 x_1\ldots \mathcal{H}^m x_m\psi_0,\\
\theta^* &= & \mathcal{H}^{m+1} x_{m+1}\ldots \mathcal{H}^{m+n}x_{m+n}\theta_0,
\end{eqnarray*}
 and  $\psi\equiv \psi^*$ and $\theta\equiv \theta^* $.  Now $\phi\vdash \psi^* \vee\theta^*$, using two 
applications of the rule \ref{rule1}. Next we prove using induction on 
$m$ that, from $\psi^* \vee\theta^*$, we can derive 
\begin{equation}\label{goalform}
 \mathcal{H}^1 x_1\ldots \mathcal{H}^mx_m\mathcal{H}^{m+1} x_{m+1}\ldots \mathcal{H}^{m+n}x_{m+n}(\psi_0\vee \theta_0).
\end{equation}
Let $m=0$. We prove this case again by induction; for $n=0$ the claim holds. Suppose that $n=l+1$. We assume that $\mathcal{H}^1=Q$. The  case   $\mathcal{H}^1=\dual Q$ is analogous, and the cases $\mathcal{H}^1\in \{\exists, \forall\}$ are handled exactly as in  \cite{Kontinen:2012}. The following deduction now shows the claim:
\begin{enumerate}
\item  $\psi_0 \vee Q x_{1}\ldots \mathcal{H}^{n}x_{n}\theta_0$ 
\item  $Q x_{1}\ldots \mathcal{H}^{n}x_{n}\theta_0\vee \psi_0 $ (rule \ref{rule2})
\item $Q x_{1}(\mathcal{H}^{2} x_{2}\ldots \mathcal{H}^{n}x_{n}\theta_0\vee \psi_0) $ (rule \ref{scopeQ})
\item $Qx_{1}\ldots \mathcal{H}^{n}x_{n}( \psi_0\vee \theta_0)$ (rule \ref{rule:mon}  and D1),  
\end{enumerate}
where D1 is the derivation 
\begin{enumerate}
\item $\mathcal{H}^{2} x_{2}\ldots \mathcal{H}^{n}x_{n}\theta_0\vee \psi_0$
\item      \hspace{1cm}     .
\item        \hspace{1cm}    .
\item       \hspace{1cm}     .
\item $ \mathcal{H}^{2} x_{2}\ldots \mathcal{H}^{n}x_{n}(\theta_0\vee \psi_0)$ (induction assumption)
\item      \hspace{1cm}     .
\item        \hspace{1cm}    .
\item       \hspace{1cm}     .
\item $ \mathcal{H}^{2} x_{2}\ldots \mathcal{H}^{n}x_{n}(\psi_0\vee \theta_0)$ (D2)
%\item $Qx_{1}\ldots \mathcal{H}^{n}x_{n}(\psi_0\vee \theta_0)$ ($\exists$ I)
\end{enumerate}
where D2 is a derivation that swaps the disjuncts. This concludes the proof for the case $m=0$.

Assume then that $m=k+1$ and that the claim holds for $k$. Now the following derivation shows the claim. Again we consider only the case $\mathcal{H}^1=Q$.

\begin{enumerate}
\item  $Qx_1\mathcal{H}^2 x_2\ldots \mathcal{H}^mx_m\psi_0 \vee  \mathcal{H}^{m+1}x_{m+1}\ldots \mathcal{H}^{m+n}x_{m+n}\theta_0$
\item  $Qx_1(   \mathcal{H}^2 x_2\ldots \mathcal{H}^mx_m\psi_0 \vee  \mathcal{H}^{m+1}x_{m+1}\ldots \mathcal{H}^{m+n}x_{m+n}\theta_0)$ (rule \ref{scopeQ})
\item $Q x_1\ldots \mathcal{H}^mx_m\mathcal{H}^{m+1} x_{m+1}\ldots
	\mathcal{H}^{m+n}x_{m+n}(\psi_0\vee\theta_0)$ (rule \ref{rule:mon} and D3)
\end{enumerate}
where  D3 is the following derivation given by the induction assumption: 
\begin{enumerate}
\item  $\mathcal{H}^2 x_2\ldots \mathcal{H}^mx_m\psi_0 \vee  \mathcal{H}^{m+1}x_{m+1}\ldots \mathcal{H}^{m+n}x_{m+n}\theta_0$
\item      \hspace{1cm}      .
\item       \hspace{1cm}     .
\item         \hspace{1cm}   .
\item $ \mathcal{H}^2 x_2\ldots \mathcal{H}^mx_m\mathcal{H}^{m+1} x_{m+1}\ldots \mathcal{H}^{m+n}x_{m+n}(\psi_0\vee\theta_0)$ 
%\item  $Q x_1\ldots \mathcal{H}^mx_m\mathcal{H}^{m+1} x_{m+1}\ldots \mathcal{H}^{m+n}x_{m+n}(\psi_0\vee\theta_0)$ ($\exists$ I)
\end{enumerate}
This concludes the proof.

\item[Step 2.] The next step is to show that from a quantifier-free formula $\theta$ it is possible to derive an equivalent formula of the form: 
\begin{equation}\label{3}
\exists z_1\ldots \exists z_n(\bigwedge_{1\le j\le n}\dep(\vec{x}^j,z_j)\wedge \theta^*),   
\end{equation}
where $\theta^*$ is a quantifier-free formula without dependence atoms. Again the claim is proved  using induction on $\theta$ using in particular rule  \ref{rule6}.  Note that the quantifier $Q$ does not play any role in this step, hence the claim can be proved exactly as in  \cite{Kontinen:2012}. 

\item[Step 3.] The deductions  in Step 1 and 2 can be combined  (from $\phi$ to \eqref{2}, and then from $\theta$ to  \eqref{3}) to show that  
\begin{equation}\label{4}
\phi\vdash \mathcal{H}^1 x_1\ldots \mathcal{H}^mx_m \exists z_1\ldots \exists z_n(\bigwedge_{1\le j\le n}\dep(\vec{x}^j,z_j)\wedge \theta^*).   
\end{equation}
Note that for $\mathcal{H}_i=Q$, rule \ref{rule:mon} is needed in this deduction.

\item[Step 4.] We transform the $\mathcal{H}^i$-quantifier prefix in \eqref{4} to the required form (see \eqref{NF})  by using rule \ref{rule:depi} and pushing the new dependence atoms as new conjuncts to 
\begin{equation}\label{5}
\bigwedge_{1\le j\le n}\dep(\vec{x}^j,z_j).
\end{equation}
%Note that each swap of the quantifier $\exists $ with a $Q$, $\dual Q$ or $\forall$ gives rise to a new dependence atom   %$\dep(\vec{x}_i,x_j)$ which we then push to the quantifier-free part of the formula.

We prove the claim using induction on the length $m$ of the $\mathcal{H}$-quantifier block in \eqref{4}. For $m=0$ the claim holds. Suppose that the claim holds for $k$ and $m=k+1$. We consider first the case $\mathcal{H}^1=Q$.
%as the case $\mathcal{H}=\dual Q$ is analogous and the cases $\mathcal{H}\in \{\forall, \exists\}$ are analogous to  \cite{Kontinen:2012}.
The following derivation now shows the claim:
\begin{enumerate}
\item$ Q x_1 \mathcal{H}^2x_2\ldots \mathcal{H}^mx_m \exists z_1\ldots \exists z_n(\bigwedge_{1\le j\le n}\dep(\vec{x}^j,z_j)\wedge \theta^*)$ 
\item  $Q x_1\mathcal{H}^{i_1} x_{i_1}\cdots \mathcal{H}^{i_h}  x_{i_h}\exists \vec{x}'  \exists \vec{z} (\bigwedge_{1\le j\le n'}\dep(\vec{x}^j,w_j)\wedge \theta^*)$ (rule \ref{rule:mon} and D4)
\end{enumerate}
where $\mathcal{H}_{i_j}$, for $1\le j\le h$,  is either $Q$, $\dual Q$ or $\forall$, and  D4 is the following derivation that exists by the induction assumption:
\begin{enumerate}
\item $ \mathcal{H}^2 x_2\ldots \mathcal{H}^mx_m \exists z_1\ldots \exists z_n(\bigwedge_{1\le j\le n}\dep(\vec{x}^j,z_j)\wedge \theta^*)$ 
\item       \hspace{1cm}       .
\item       \hspace{1cm}       .
\item       \hspace{1cm}       .
\item $\mathcal{H}^{i_1} x_{i_1}\cdots \mathcal{H}^{i_h}  x_{i_h}\exists \vec{x}'  \exists \vec{z} (\bigwedge_{1\le j\le n'}\dep(\vec{x}^j,w_j)\wedge \theta^*)$
\end{enumerate}
The case $\mathcal{H}^1=\forall$ can be proved analogously.  Next we consider the case $\mathcal{H}^1=\exists $ and $\mathcal{H}^i=\exists $ for all $2\le i \le m$. In this case the quantifier  $\mathcal{H}^1$ is already in the right place in the quantifier prefix. We will record the variables determining 
$\mathcal{H}^1$  by a new dependence atom and then move it to the quantifier free part of the formula. This is done because
 each existentially quantified variable is determined by one and only one dependence atom in the normal form  \eqref{NF}. 
We will use the following auxiliary derivation D5:
\begin{enumerate}
\item $\exists x \chi(\vec{z})$
\item $\exists x\forall y \chi(\vec{z})$ ($\exists $ E and D6)
\item $\forall y\exists x(\dep(\overline{z},x)\wedge \chi(\vec{z}))$ (rule \ref{rule6}) 
\item $\exists x(\dep(\overline{z},x)\wedge \chi(\vec{z}))$ ($\forall $ E),
\end{enumerate}
where D6 refers to the following derivation
\begin{enumerate}
\item $\chi(\vec{z},x)$
\item $\forall y \chi(\vec{z},x)$ ($\forall$ I with $y$ a fresh variable)
\item $\exists x\forall y \chi(\vec{z})$ ($\exists$ I)  
\end{enumerate}
Let us now prove the case $\mathcal{H}^1=\exists$ with $\mathcal{H}^i=\exists $ for all $2\le i \le m$.
\begin{enumerate}
\item $\exists x_1 \exists x_2\ldots \exists x_m \exists z_1\ldots \exists z_n(\bigwedge_{1\le j\le n}\dep(\vec{x}^j,z_j)\wedge \theta^*)$ 
\item  $\exists x_1 \exists x_2\ldots \exists x_m \exists z_1\ldots \exists z_n( \bigwedge_{1\le j\le n+m-1}\dep(\vec{x}^j,w_j) \wedge \theta^*)$ ($\exists$ E and D7) % $w_j=z_j$ for $j\le n$, and  $w_j=x_{(j+1)-n}$ if $j> n$)
\item $\exists x_1(\dep(\vec{x}^{n+m},x_1) \wedge \exists x_2\ldots \exists x_m \exists \vec{z}(\bigwedge_{1\le j\le n+m-1}\dep(\vec{x}^j,w_j)\wedge \theta^*))$ (D5) 
\item  $\exists x_1 \exists x_2\ldots \exists x_m \exists z_1\ldots \exists z_n(\bigwedge_{1\le j\le n+m}\dep(\vec{x}^j,w_j)\wedge \theta^*)$  (D8)
\end{enumerate}
where D7 is a derivation that exists by the induction assumption, and D8 is a derivation that pushes $\dep(\vec{x}^{n+m},x_1)$ into the quantifier free part of the formula. The case $\mathcal{H}^1=\exists $, where  
$\mathcal{H}^i=\forall$ for some $2\le i \le m$ can be proved similarly to  \cite{Kontinen:2012} adding one additional trasformation 
in which the redundant dependence atoms (created not by the first swap of $\mathcal{H}^1$ with $\forall$, $Q$, or $\dual Q$) 
are deleted from the formula using essentially the rule $\wedge$ E. 
\end{itemize}

Steps 1-4 show that from a formula $\phi$ a formula of the form can be deduced 
\begin{equation}\label{6}
\mathcal{H}^1 x_1\ldots \mathcal{H}^m x_m \exists y_1\ldots \exists y_n \bigl(\bigwedge _{1\le 
	j\le
n}\dep(\overline{x}^i,y_i)\wedge\theta\bigr),
\end{equation}
where $\mathcal{H}^i$ is either $Q$, $\dual Q$ or $\forall$ and $\theta$ is a quantifier-free $\FO$-formula.
 Furthermore, $\phi$ and the formula in \eqref{6} are logically equivalent since
 logical equivalence is  preserved in each of the Steps 1-4.
\end{proof}

\subsection{Completeness for $\df(Q,\dual Q)$}

In this section we prove a completeness result for $\df(Q,\dual Q)$ with respect to $\FO(Q,\dual Q)$
consequences of $\df(Q,\dual Q)$-sentences, with weak semantics. 
Analogously to \cite{Kontinen:2012}, we approximate  
$\df(Q,\dual Q)$-sentences in the normal form \eqref{NF} by an infinite set of $\FO(Q,\dual Q)$ sentences. We use an extra predicate $R$ to encode a team witnessing the satisfiability of the quantifier prefix  $\mathcal{H}^1 x_1\ldots \mathcal{H}^m x_m$.
% we may define \emph{finite}
%approximations of any $\df(Q,\dual Q)$ sentence written in normal form . 

Let $\sigma$ be 
$$
\H^1 x_1\ldots \H^m x_m \exists y_1\ldots \exists y_n \bigl(\bigwedge _{1\le i \le
n}\dep(\overline{x}^i,y_i)\wedge\theta(x_1,\ldots,x_m,y_1,\ldots,y_n)\bigr),
$$
where each $\H^i$ is either $Q$, $\dual Q$ or $\forall$. 

We define finite approximations $A^i\sigma$ of $\sigma$ as follows. The first approximation,
$A^1\sigma$, is
$$ \forall x_1 \ldots \forall x_m \exists y_1 \ldots \exists y_n \bigl(R(x_1,\ldots,x_m) \imp
\theta(x_1,\ldots,x_m,y_1,\ldots,y_n)\bigr),$$
or in compressed form:
$$  \forall \bar x \exists \bar y \bigl(R(\bar x) \imp \theta(\bar x, \bar y)\bigr).$$

The second approximation $A^2\sigma$ is 
\begin{multline*} \forall \bar x_1 \exists \bar y_1 \forall \bar x_2 \exists y_2 \bigl(
R(\bar x_1) \land R(\bar x_2) \imp 
\theta(\bar x_1,\bar y_1) \land  
\theta(\bar x_2,\bar y_2)  \land {} \\
\bigwedge_{1 < i < n} (\bar x_1^i=\bar x_2^i \imp y_{i,1}=y_{i,2}) 
\bigr)
\end{multline*}
With the notational convention that $(x_{i_1},\ldots,x_{i_k})_j$ is the sequence    $(x_{i_1,j},\ldots,x_{i_k,j})$. 
By generalizing this construction we get the $k$:th approximation: 
\begin{multline*}
	\forall \bar x_1 \exists \bar y_1 \ldots \forall \bar x_k \exists \bar y_k 
	\bigl( \bigwedge_{1 \leq j \leq k}R(\bar x_j) \imp 
	\bigwedge_{1 \leq j \leq k} \theta(\bar x_j,\bar y_j) \land {} \\  
	\bigwedge_{\substack{1 \leq i \leq n \\1 \leq j,j' \leq k}} (\bar x_j^i=\bar x_{j'}^i \imp y_{i,j}=y_{i,j'}) \bigr)
\end{multline*}

Also we need a sentence saying that $R$ is of the right kind, witnessing the quantifier prefix:
Let $B\sigma$ be $$\H^1 x_1\ldots \H^m
x_m R(x_1,\ldots,x_m).$$

We will adopt the following approximation rule in our deduction system:

\begin{prooftree}
\AxiomC{$\sigma$}
\AxiomC{$[B\sigma]$}
\noLine
\UnaryInfC{$\ddots$}
\AxiomC{$[A^n\sigma]$}
\noLine
\UnaryInfC{$\iddots$}
\noLine
\BinaryInfC{$\psi$}
\RightLabel{(Approx)}
\BinaryInfC{$\psi$}
\end{prooftree}
where $\sigma$ is a sentence in normal form,  and 
$R$ does not appear in $\psi$ nor in any uncancelled assumptions in the derivation of $\psi$, except for $B\sigma$ and $A^n\sigma$.

\begin{lemma} \label{lem:approxsound}
	Adding the approximation rule to the inference system results in a sound system for $\df(Q,\dual Q)$ with regard to weak semantics.
\end{lemma}
\begin{proof}
	We plug in the following induction step to the proof of Proposition \ref{soundness}:

	Assume that there is a derivation of $\psi$ from $\Gamma$ ending with
	the approximation rule. Then there are shorter derivations from $\Gamma$ of $\sigma$ and from
	$\Gamma',B\sigma,A^n\sigma$ of $\psi$, where $\Gamma' \subseteq \Gamma$ is such that $R$ does not occur in $\Gamma'$. By the induction hypothesis we get $\Gamma \models \sigma$
	and $\Gamma',B\sigma,A^n\sigma\models \psi$. We will prove that $\Gamma \models \psi$, by
	assuming $\M,X \models \Gamma$ for some non-empty $X$ and proving that $\M,X \models \psi$. 

	Assume $\sigma$ is of the form 
$$
\H^1 x_1\ldots \H^m x_m \exists y_1\ldots \exists y_n \bigl(\bigwedge _{1\le i \le
n}\dep(\overline{x}^i,y_i)\wedge\theta\bigr).
$$
where $\theta$ is a quantifier free first order formula.

From the fact that $\M,X \models \sigma$ we get $\M \models \sigma$ and thus there is a
(non-empty) team $Y$ such that 
$$\M,Y \models \exists y_1\ldots \exists y_n \bigl(\bigwedge _{1\le i \le
n}\dep(\overline{x}^i,y_i)\wedge\theta\bigr).$$
Let $r \subseteq M^m$ be the relation $Y(\bar x)$ corresponding to $Y$. Then $(\M,r)\models
B\sigma$ and, it should also be clear that $(\M,r) \models A^n\sigma$.
Since $R$ does not occur in $\Gamma'$ we have  $(\M,r),X \models \Gamma',B\sigma,A^n\sigma$. By the induction hypothesis $(\M,r),X \models
\psi$, and since $R$ does not occur in $\psi$ we have  $\M,X\models \psi$. 
\end{proof}

The main result of this section can now be stated as follows.
\begin{theorem}\label{mainthm} Let $T$ be a set of sentences of $\df(Q,\dual Q)$ and
	$\phi\in \FO(Q,\dual Q)$ a sentence. Then the following are equivalent:
\begin{description}
\item[(I)] $T\models_w\phi$
\item[(II)] $T\vdash\phi$
\end{description}
\end{theorem}

The following lemmas are needed in the proof.

\begin{lemma}\label{lemrec}
	If $T$ is a set of $\FO(Q,\dual Q)$-sentences consistent in the deduction system described above then there are a countable recursively saturated
	model $\M$ and an interpretation $q$ of $Q$ such that $(\M,q,\dual q) \models T$.
\end{lemma}
\begin{proof}
	First translate $T$ to $T^\lnot$ in which each $\dual Qx\phi$ is replaced by $\lnot Qx \lnot \phi$.  
	By using the same argument as in \cite{keisler1970logic, kaufmann1985quantifier} we may reduce $\FO(Q)$ to $\FO$ by replacing subformulas of the form $Qx\phi$ with new relation symbols
	$R_\phi(\bar y)$, $\bar y$ being
	the free variables of $Qx\phi$. This will reduce the set $T^\lnot$ to a set $T^\ast$. 
	Let $T'$ be $T^\ast$ together with the translations of the universal closures of 
	\begin{itemize}
		\item $(\varphi \imp \psi) \imp (Q x \varphi \imp Q x \psi)$, for all $\varphi$ and
			$\psi$; and
		\item $Qx \phi \imp Qy (\phi[y/x])$, for all $\phi$ such that the substitution is legal.
	\end{itemize}
             Now $T'$ is consistent by the same argument as in \cite{kaufmann1985quantifier}.
	Let $\M^\ast$ be a countable recursively saturated model of $T'$, and $M$ its reduct to the
	original signature. Now we may define $q$ to be  
	$$ \set{A \subseteq M| A \supseteq \set{a \in M | \M^\ast, s[a/x] \models \phi^\ast} \text{ for some $\phi$ s.t.
		}\ M^\ast,s \models
	Qx\phi^\ast}.$$
	Proposition 2.3.4 in \cite{kaufmann1985quantifier} shows that $(\M,q) \models T^\lnot$, and thus
	$(\M,q,\dual q) \models T$.
\end{proof}

\begin{lemma}\label{lemrec2}
	In a countable recursively saturated weak model $(\M,q,\dual q)$ in which  $B\sigma$ and $A^n\sigma$ holds for all
$n$, $\sigma$ holds. 
\end{lemma}

\begin{proof}
	Suppose $\sigma$ is
$$
\H^1 x_1\ldots \H^m x_m \exists y_1\ldots \exists y_n \bigl(\bigwedge _{1\le i \le
n}\dep(\overline{x}^i,y_i)\wedge\theta\bigr).
$$
Note that the sentences $A^n\sigma$ can be viewed as the finite approximations as defined in
\cite{Kontinen:2012} (and see also \cite{barwise1976some}) of the $\df$ sentence $\sigma'$:
$$\forall \bar x \exists \bar y \bigl(\bigwedge _{1\le i \le
n}\dep(\overline{x}^i,y_i)\wedge(R(\bar x) \imp \theta)\bigr).$$
Thus by Theorem 2.4 in \cite{barwise1976some} (see also \cite{Kontinen:2012}), we know that $\M \models \sigma'$.

Let $X$ be the team $\set{s: \set{x_1,\ldots,x_k} \to M | (s(x_1),\ldots,s(x_m)) \in R^\M}$.
To prove that $(\M,q,\dual q) \models \sigma$ we find $F_1, \ldots, F_m$ so that
$\set{\emptyset}[F_1/x_1]\ldots[F_m/x_m] = X.$

\begin{align*}
	F_1(\emptyset) &= X \upharpoonright \set{x_1},\\ 
	F_{i+1}&: \set{\emptyset}[F_1/x_1]\ldots[F_{i}/x_{i}] \to M\\
  F_{i+1}(s) &= \set{a \in M | \exists s' \in X:
(s'(x_1),\ldots,s'(x_{i+1}))=(s(x_1),\ldots,s(x_i),a)}.
\end{align*}
By the assumption $(\M,q,\dual q)\models B\sigma$ it follows that $F_i(s) \in \H^{i}_M$. Furthermore, since  $\M \models \sigma'$ we get  that
$$\M,X \models \exists y_1\ldots \exists y_n \bigl(\bigwedge _{1\le i \le
n}\dep(\overline{x}^i,y_i)\wedge\theta\bigr).
$$
Therefore
$$(\M,q,\dual q)\models \H^1 x_1\ldots \H^m x_m \exists y_1\ldots \exists y_n \bigl(\bigwedge _{1\le i \le
n}\dep(\overline{x}^i,y_i)\wedge\theta\bigr)$$
as wanted.
\end{proof}

\begin{proof}[Proof of Theorem \ref{mainthm}.]
	(I) $\Rightarrow$ (II): This is just a special case (for sentences) of soundness.

	(II) $\Rightarrow$ (I):  Suppose $T \nvdash \phi$, where $\phi$ is a $\FO(Q,\dual Q)$-sentence. 
	We will construct a weak model of $T \cup \set{\lnot \phi}$ showing that $T \nvDash_w \phi$.
	Replacing $T$ with the set $T'=\set{B\sigma, A^n\sigma | \sigma \in T, n \in \mathbb N}$ we can
	conclude that $T' \cup \set{\lnot \phi} \nvdash \bot$. 
By applying Lemma \ref{lemrec} we get 
	 a weak countable recursively saturated model $(M,q,\dual q)$ of $T'$.
	Lemma \ref{lemrec2} implies that $(M,q,\dual q) \models T$. Now since $(M,q,\dual q) \nvDash
	\phi$, we get  $T\not \models_w \phi$ as wanted.
\end{proof}

\subsection{Completeness for $\df(Q_1,\dual{Q_1})$}

We will now prove a completeness result similar to Theorem \ref{mainthm} for the logic
$\df(Q,\dual Q)$ where $Q$ is interpretated as $Q_1$, the quantifier ``there exists uncountably
many.'' In this section we consider only structures over uncountable universes.

We add the following two rules from \cite{keisler1970logic}  to the system presented in Section 3.
Note that the approximation rule of section 4.2 is not included. 

$$\infer{\lnot Qx (x = y \lor x = z)}{}$$

$$\infer{\exists y Qx \varphi \lor Qy\exists x \varphi}{Qx \exists y \varphi}$$

The intuitive meaning of the second rule is that a countable union of countable sets is countable.
The first is needed to avoid $Q$ being interpreted as the quantifier ``the exists at least
two.''

For each $\df(Q,\dual Q)$ sentence $\sigma$
	$$\H^1 x_1\ldots \H^m x_m \exists y_1\ldots \exists y_n \bigl(\bigwedge _{1\le i \le
n}\dep(\overline{x}^i,y_i)\wedge\theta\bigr)
 $$
	in normal form we define the Skolem translation $S\sigma$ of
	$\sigma$ to be:
	$$
\H^1 x_1\ldots \H^m x_m   \theta(f_i(\bar x^i)/y_i) ,
	$$
	where the $f_i$'s  are new function symbols of the right arity. 
If $\sigma$ is a sentence in the signature $\tau$ then $S\sigma$ will be in the extended signature
$\tau \cup \set{f_1, \ldots, f_n}$. 

The last rule of the deduction system is the following:

\begin{prooftree}
\AxiomC{$\sigma$}
\AxiomC{$[S\sigma]$}
\noLine
\UnaryInfC{$\vdots$}
\noLine
\UnaryInfC{$\psi$}
\RightLabel{(Skolem)}
\BinaryInfC{$\psi$}
\end{prooftree}

Here $\sigma$ is a $\df(Q,\dual Q)$ sentence in normal form, and the function symbols
$f_1,\ldots,f_n$ do not occur in $\psi$ nor in any uncancelled assumption of the derivation of $\psi$, except for $S\sigma$.

\begin{proposition}\label{prop:sound2}
	If $T \vdash \phi$ in the deduction system for $\df(Q_1,\dual{Q}_1)$ then $T \vDash \phi$. 
\end{proposition}

\begin{proof}
	We extend the proof of Proposition \ref{soundness} to also cover the three new rules:

	(1) The soundness of the first rule is easily seen by observing that the formula $Q_1x (x=y \lor
	x=z)$ is a $\FO(Q_1)$ formula and thus a team satisfies it iff every assignment in the team
	satisfies the formula. 

(2) For the second rule we need to prove that if $\M,X \models \Gamma,Q_1x \exists y \phi$ then $\M,X
\models \Gamma,\exists y Q_1x \phi \lor Q_1y \exists x \phi$. By the assumption we get  functions
$F: X \to Q_M$ and $f: X[F/x] \to M$ such that $\M,X[F/x][f/y] \models \phi$. 
Thus, for each $s \in X$ there is a binary relation $R_s=\set{(a,f(s[a/y])) | a\in F(s)}$ such that 
$(M,R_s) \models Q_1x \exists y R(x,y)$. 
Let $$Y =\set{s \in X | (M,R_s) \models \exists yQ_1x R(x,y) }$$ and 
$$Z =\set{s \in X | (M,R_s) \models Q_1 y \exists x R(x,y) }.$$ By the validity of the rule for $\FO(Q_1)$ we see that $X = Y \cup Z$.

It should be clear that $Y \models \exists y Q_1x \phi$ since by letting $g(s)$ be such that $(M,R_s) \models Q_1 x R(x,g(s))$ and
$$G(s[g(s)/y]) = \set{a \in M | (M,R_s) \models R(a,g(s))},$$ we have that
$Y[g/y][G/x] \subseteq X[F/x][f/y]$ and thus by downward closure 
$$\M,Y[g/y][G/x] \models \phi.$$

Similarly we can prove that $\M,Z \models Q_1y \exists x \phi$, and thus that $\M,X
\models \Gamma,\exists y Q_1x \phi \lor Q_1y \exists x \phi$.

(3) For the Skolem rule assume that there is a derivation of $\psi$ from $\Gamma$ ending with
	the Skolem rule. Then there are shorter derivations from $\Gamma$ of $\sigma$ and from
	$\Gamma,S\sigma$ of $\psi$. By the induction hypothesis we get $\Gamma \models \sigma$
	and $\Gamma,S\sigma\models \psi$. We will prove that $\Gamma \models \psi$, by
	assuming $\M,X \models \Gamma$ for some non-empty $X$ and proving that $\M,X \models \psi$. 

	From the proof of Theorem 3.5 in \cite{Engstrom.Kontinen:2013} we see that $\M \models \sigma$ iff 
$\M \models \exists f_1\ldots\exists f_k S\sigma$. From $\M,X \models \Gamma$ and $\Gamma \models \sigma$ we get that $\M \models \sigma$ and thus there are $f_1,\ldots,f_k$ such that $(M,f_1,\ldots,f_k) \models S\sigma$. 
Now since  the $f_i$'s do not occur in formulas $\Gamma'\subseteq \Gamma$ used in the derivation of $\psi$,  $\M,X \models \Gamma$ implies that 
 $(M,f_1,\ldots,f_k),X \models \Gamma'$. By locality, we also have  $(M,f_1,\ldots,f_k),X \models S\sigma$. Therefore,
by the induction hypothesis, we get that $(M,f_1,\ldots,f_k),X \models \psi$, and, since  the $f_i$'s do not occur in $\psi$,   $\M,X \models \psi$ follows.
%From the assumption $\Gamma,S\sigma \models \psi$ we can conclude that $(M,f_1,\ldots,f_k),X \models \Gamma,\psi$ and %since the $f_i$'s do not occur in $\Gamma$ nor in $\psi$ we have that $\M,X \models \Gamma,\psi$. 
\end{proof}

\begin{theorem}\label{comp2}
If $T$ is a set of $\df(Q_1,\dual Q_1)$ sentences and $\phi$ is a $\FO(Q_1,\dual Q_1)$ sentence then $T \vdash \phi$ iff $T \models \phi$.
\end{theorem}

\begin{proof}
Assume $T \nvdash \phi$. We build a model of $T'=T \cup \set{\lnot \phi} \nvdash \bot$ by translating sentences $\sigma$ of $T$ into normal form $\sigma_{\text{nf}}$ and considering the $\FO(Q,\dual Q)$ theory $T_S=\set{ S \sigma_{\text{nf}}| \sigma \in T} \cup \set{\lnot \phi}$. This theory is consistent, since otherwise the Skolem rule and Proposition \ref{prop:nf} would allow us to derive a contradiction from $T'$.

Since the deduction system for $\df(Q_1,\dual Q_1)$ contains Keisler's system \cite{keisler1970logic} we may apply the completeness theorem for $\FO(Q_1)$ and get a model $\M$ of $T_S \cup \set{\lnot \phi}$. 
By the remark made in the proof of Proposition \ref{prop:sound2} and Proposition \ref{prop:nf} $\M$ is also a model of $T \cup \set{\lnot \phi}$. Thus $T \nvDash \phi$.
\end{proof}

\section{Conclusion}

In this article we have presented inference rules and axioms for extensions of dependence logic by monotone generalized quantifiers.  We also proved two completeness results for $\FO(Q)$ consequences in the cases where  $Q$  either has a  weak interpretation or $Q$ it is interpreted as ``there exists uncountable many.'' In the first completeness theorem,  an important feature of the proof is the approximation of a  $\df(Q_1,\dual{Q}_1)$ sentence by an infinite set of $\FO(Q)$ sentences. In the second completeness theorem the approximations were replaced by the  Skolem rule which however is slightly unsatisfactory due to the extra function symbols $f_i$ used in its formulation. In future work our plan is to further analyze the completeness theorem of $\df(Q_1,\dual{Q}_1)$,  and  replace the Skolem rule with rules that do not rely on the explicit use of the Skolem functions $f_i$. 

\section{Acknowledgements}
The second and the third author were supported by grants 264917 and 251557 of the Academy of Finland. The first author was supported by the Swedish Research Council.

\vspace{1cm} \noindent
\begin{tabular}{l}
{\it Fredrik Engstr\"om}\\
Department of Philosophy, Linguistics and Theory of Science\\
University of Gothenburg, Sweden\\
\tt{fredrik.engstrom@gu.se}\\
\\
{\it Juha Kontinen}\\
Department of Mathematics and Statistics\\
University of Helsinki, Finland\\
\tt{juha.kontinen@helsinki.fi}\\
\\
{\it Jouko V\"a\"an\"anen}\\
Department of Mathematics and Statistics\\
University of Helsinki, Finland\\
and\\
Insitute for Logic, Language
 and Computation\\
University of Amsterdam,
 The Netherlands\\
\tt{jouko.vaananen@helsinki.fi}\\
\end{tabular}

\bibliographystyle{plain}

\bibliography{ref_compl}

\end{document}